\documentclass[11pt,a4paper]{article}
\usepackage{amssymb}
\usepackage{amsmath}

\newcommand{\qed}{
  \ifmmode
   \eqno{\qedsymbol}
  \else
    \leavevmode\unskip\penalty9999 \hbox{}\nobreak\hfill\hbox{\qedsymbol}
  \fi
}
\newcommand{\qedsymbol}{\leavevmode\vrule height 1.2ex width 1.1ex depth -.1ex}
\newenvironment{proof}{\begin{trivlist}\item[\hskip
\labelsep{\bf Proof.\quad}]}
{\hfill\qed\rm\end{trivlist}}

\newtheorem{theorem}{Theorem}
\newtheorem{corollary}[theorem]{Corollary}
\newtheorem{proposition}[theorem]{Proposition}
\newtheorem{lemma}[theorem]{Lemma}

\def\n{\mathbb{N}}
\def\z{\mathbb{Z}}

\def\mod{\hbox{ mod }}
\newcommand{\ndiv}{\,\hbox to 0.42em{$\hspace{-0.33em}\not|\hspace{0.17em}$}}

\title{A Totient Function Associated with Variants of Groups}
\author{James Renshaw\\
School of Mathematical Sciences,\\
University of Southampton,\\
Southampton, SO17 1BJ, England.\\
ORCID: 0000-0002-5571-8007\\
Email: j.h.renshaw@soton.ac.uk}
\date{March 2026}

\begin{document}
\maketitle

\begin{abstract}
Motivated by an application of semigroup variants to the discrete log problem in groups and related cryptographic applications, we introduce a new kind of totient function, related to both Euler's function and a generalisation of Euler's function introduced in 1869 by Schemmel. We focus on the problem of how to evaluate this function, and the number theory involved, while non-trivial and at times slightly technical, is reasonably accessible to a wide audience. It should also become clear that there are obvious generalisations of his new function that the interested reader might like to pursue.
\end{abstract}

\parindent=0pt
MSC: 11A05, 11A25, 20A05, 20M10


\section{Introduction}\label{introduction-section}
Suppose we download the contents of this article in the form of a PDF file and we wish to encrypt its contents. There are of course many well known and effective encryption schemes available, but we wish to consider a new one based on {\em variants of groups} and a new type of totient function related to Euler's function. Let us define this function $T:\n\to\n$ by setting $T(n)$ to be the number of odd units $m\in\z_n$ such that $(m-1)/2$ is also a unit, and where $\z_n$ is the ring of integers modulo $n$. Recall that a {\em unit} is a positive integer $m \in\z_n$ such that there exists $m'\in\z_n$ with $mm'\equiv1\mod(n)$ or equivalently such that $\gcd(m,n)=1$. We shall see how the value of $T(n)$ relates to the security of our system.

As a specific example, let $p$ be a {\em safe prime}, that is to say a prime of the form $p=2q+1$ where $q$ is also prime (such a prime $q$ is often referred to as a {\em Sophie Germain prime}, and there are conjectured to be an infinite number of these). Let us first encode our data using values from $U_{p}$, the group of units of $\z_p$. Choose an {\em encryption key} that consists of the pair $(x,e)$ where $x\in U_{p}$, and $e \in U_{p-1}$. If $g \in U_{p}$ is our {\em plaintext} then the {\em ciphertext} is defined to be
$$
c=\left(gx\right)^{e-1}g.
$$
A {\em known plaintext attack} assumes that we know the values of both $g$ and $c$, but if, as part of such an attack, we are reduced to a `brute force' or trial multiplication solution of this {\em discrete log problem}, we potentially have to test every candidate for $x$ and every candidate for $e$. Now $|U_p| = p-1=2q$ and $|U_{p-1}| = q-1$ and thus we have $2q(q-1)\sim O(p^2)$ pairs $(x,e)$ to check and the problem becomes infeasible for large enough values of $p$. However, there is another issue, of more importance to us, that is summarised by the following theorem.

\begin{theorem}\label{motivating-theorem}
Let $G$ be a finite group of even order $n$ and let $g,x\in G$ and $e\in U_{n}$. Then the number of distinct pairs $(y,f)$ with $y\in G, f\in U_n$ such that $\left(gy\right)^{f-1} = \left(gx\right)^{e-1}$ is at least equal to $T(n)$.

In particular if $p=2q+1$ is a safe prime and $G=U_{p}$ then the number of such pairs is at least $(p-3)/4$ when $q\equiv1\mod(4)$ and at least $(p-7)/4$ when $q\equiv3\mod(4)$.
\end{theorem}
It should be clear that the above encryption scheme and theorem can be applied to any suitable finite group, and we shall say more about the cryptographic background together with an outline of the proof of the first part of Theorem~\ref{motivating-theorem}, as well as an interesting stronger extension of the second part of this theorem, in Section~\ref{motivation-section}. For the moment, we are primarily interested in the function $T(n)$, since the above theorem says that when we work through the $O(p^2)$ pairs $(y,f)$, computing $(gy)^{f-1}$ and checking that value against $(gx)^{e-1}$, there will be at least $T(n)$ such pairs that provide a match, but only one of which is the `correct' solution $(x,e)$. It would clearly therefore be advantageous to ensure that $T(n)$ is suitably large relative to $n$. Metaphorically speaking, we are looking for a specific needle in a large haystack full of identical needles.

\smallskip

For more details of the use of groups and the discrete log problem in modern cryptography, see for example~\cite{hoffstein}. For some basic results and notation in elementary number theory, see any introductory text on the subject, such as~\cite{jones-2005}.

\section{The totient function}\label{totient-section}

We define the function $T:\n\to\n$ as
$$
T(n)=\left|\{1\le m\le n\;|\; \gcd(m,n)=\gcd((m-1)/2,n)=1\}\right|.
$$
For completeness we define$\footnote{There is an argument that $T(1)=1$ might be more appropriate. However as we are principally interested in large values of $n$ we are content, for now, to let $T(1)=0$.}\ $T(1) = 0. It is easy to demonstrate that $T$ is not multiplicative, as for example $T(2) = 0, T(5) = 1$ ($m=3$ is the only suitable unit) whereas $T(10)=2$ ($m=3$ and $7$), and it appears that computing a closed formula for all values of $n$ seems to be rather complicated. However we will manage to do so for many values of $n$ and produce relatively small bounds on the values of the others.

\medskip

Our totient function $T$ is a generalisation of, and related to, two others, the first of which is quite well-known. Euler's totient function, normally denoted $\phi$, counts the number of units in $\z_n$, the ring of integers modulo $n$. It is well known that $\phi$ is {\em multiplicative}, in the sense that $\phi(n)\phi(m) = \phi(nm)$ whenever $n$ and $m$ are coprime, and it is then an easy matter to determine that
$$
\phi(n) = n\prod_{p|n}\left(1-\frac1p\right)
$$
where $p$ runs through all distinct prime divisors of $n$. The second totient function related to $T$ is a function introduced in 1869 by Schemmel \cite{schemmel}. Schemmel's totient number, $S_r(n)$ counts the number of consecutive terms $1\le m,m+1,\ldots,m+(r-1)\le n$ which are all coprime to $n$, or in other words, the number of $r$ consecutive units in $\z_n$. It is easily shown that this function is also multiplicative and that
$$
S_r(n) = n\prod_{p|n}\left(1-\frac rp\right).
$$
To simplify the notation, we let
$$
S(n) = S_2(n) = n\prod_{p|n}\left(1-\frac2p\right),
$$
and since $S$ is multiplicative, it is convenient to let $S(1)=1$.

It is worth noting at this point, for future reference, that if $p_1,\ldots,p_\omega$ are the distinct prime divisors of $n>1$ then
\begin{gather}\label{-nSn-formula}
S(n) =n\left(1-\sum_{i=1}^\omega{\frac{2^1}{p_i}}+\sum_{i_1\ne i_2}{\frac{2^2}{p_{i_1}p_{i_2}}}- \ldots + (-1)^\omega\sum_{i=1}^\omega{\frac{2^{\omega}}{p_{1}\ldots p_{\omega}}}\right).
\end{gather}
We shall begin our investigation of $T(n)$ by considering the case when $n$ is odd.
\begin{theorem}\label{primepower-theorem}
Let $p$ be an odd prime and $e\ge1$ an integer. Then
$$
T(p^e) = (p^e-2p^{e-1}-1)/2 = \frac{p^e}2\left(1-\frac2p-\frac1{p^e}\right)=\frac{S(p^e)-1}2.
$$
\end{theorem}
\begin{proof}
We must count the odd integers $1\le m\le p^e$ such that $m\not\equiv0\mod p$ and $(m-1)/2\not\equiv0\mod p$. Instead, we count the number of terms that do not satisfy this condition and subtract it from $p^e$. The negation of the stated condition is
$$
(m\equiv0\mod p)\lor (m\equiv1\mod 2p)\lor (m\equiv0\mod2).
$$
There are $p^{e-1}$ elements that satisfy the first of these, $(p^{e-1}+1)/2$ that satisfy the second, and $(p^e-1)/2$ that satisfy the third. However, there are $(p^{e-1}-1)/2$ that satisfy both the first and third, and thus the number we require is
$$
p^e-p^{e-1}-\frac{p^{e-1}+1}2-\frac{p^e-1}2+\frac{p^{e-1}-1}2 = \frac{p^e-2p^{e-1}-1}2.
$$
\end{proof}
Before going further, we recall some basic binomial series. Let $k$ be a positive integer and define
$$
k'=k-\frac{1+(-1)^k}2, k''=k-\frac{1-(-1)^k}2.
$$
Then
\begin{gather*}
\frac{3^k-1}2 = \binom k1 + 2\binom k2 + \ldots + 2^{k-1}\binom k k\\
2^k = \binom k0 + \binom k1+\binom k2 + \ldots +\binom kk\\
2^{k-1}=\binom k1 + \binom k3+\ldots+\binom k{k'} = \binom k0 + \binom k2 + \ldots +\binom k{k''}\\
\end{gather*}
\begin{theorem}\label{phi-double-dash-theorem}
Let $n>1$ be an odd integer with $\omega$ distinct prime divisors. Then
$$
\left|T(n)- \frac{S(n)-1}2\right|\le\frac{3^\omega-2^{\omega+1}+1}{2}.
$$
\end{theorem}
\begin{proof}
To calculate $T(n)$ we need to remove from the set $\{1,\ldots, n\}$ those numbers $m$ such that $m$ has a non-trivial common factor with $n$, $m$ is of the form $1+2m'$ with $m'$ having a non-trivial common factor with $n$ and the even integers. In other words, after removing the even integers, for each prime $q$ with $q|n$ we remove integers of the form $kq$ and $1+lq$ for $k$ odd and $l$ even. Let $p>1$ (not necessarily prime) be an odd divisor of $n$ and define
\begin{gather*}
B_p=\{m|1\le m\le n, p|m, m\text{ odd}\},\\
C_p=\{m|1\le m\le n, p|(m-1), m\text{ odd}\},\\
A_p=B_p\cup C_p \text{ and }A=\bigcup_{p|n}A_p.
\end{gather*}
Notice that $T(n) = n-(n-1)/2-|A| = (n+1)/2-|A|$. It should also be noted that $|B_p|=|C_p|=(n+p)/2p$ and that if $p>1$ and $q>1$ are coprime odd divisors of $n$, then $B_p\cap B_q=B_{pq}$ and $C_p\cap C_q = C_{pq}$. Consequently, since $B_p\cap C_p=\emptyset$, we see that $|A_p| = n/p+1$.

We shall use the inclusion-exclusion principle to calculate $|A|$, but we first need a couple of lemmas.
\begin{lemma}\label{odd-intersect-lemma}
Let $p>1$ and $q>1$ be coprime odd divisors of the odd positive integer $n$. Then
$$
|B_p\cap C_q| =|B_q\cap C_p| = (n\pm pq)/2pq,
$$
\end{lemma}
\begin{proof}
Consider first $B_p\cap C_q$. Suppose that $px = 1+qy\le n$ for $x$ odd. We must calculate the number of solutions with $y$ even. For now, we ignore the parity of $x$ and $y$. If $(x_0,y_0)$ is the smallest non-negative solution to this system, then the general solution to the Diophantine equation is
\begin{equation}\label{diophantine}
p(x_0+qt)=1+q(y_0+pt)
\end{equation}
for $t\in\z$. It follows easily that $x_0 < q, y_0 < p$ and $y_0+pt\le (n-q)/q\le (n-1)/q$. Notice that $p\ndiv y$. If $n\ne pq$, then partition the interval  of integers $[1,n/q-1]$ into subintervals $[1,p], [p+1,2p], [2p+1,3p],\ldots, [p(n-2pq)/pq+1,p(n-pq)/pq]$, each of which contains exactly one solution $y$ to the original system. Note also that these solutions alternate between even and odd values of $y$.  Hence, the conjoined intervals $[1,2p], [2p+1,4p], \ldots, [p(n-3pq)/pq+1,p(n-pq)/pq]$ each contain exactly one {\em even} solution. If $n=pq$ then there are no such subintervals. In either case, the rest of the interval, namely  $[p(n-pq)/pq+1,n/q-1]$, also contains a solution for $y$ which {\em may} be an even solution, depending on whether the smallest positive solution $y_0$ is even or odd. Hence, there are either $(n-pq)/2pq$ or $(n+pq)/2pq$ even solutions for $y$. Calculating the solution $(x_0,y_0)$ involves an application of the Euclidean algorithm, but it is not obvious if there is an easy way to determine the parity of $y_0$ from $p$ and $q$ alone.

\smallskip

In calculating $B_q\cap C_p$, we consider the Diophantine equation $qx=1+py$ and we can arrive at a similar conclusion. Note that in this case the general solution is given by
\begin{equation}\label{diophantine_two}
q(pt-y_0) = 1+p(qt-x_0),
\end{equation}
for $t\in\z$ and $(x_0,y_0)$ is the smallest non-negative solution to equation~(\ref{diophantine}). Hence, the smallest non-negative solution for the Equation~(\ref{diophantine_two}) is $(p-y_0,q-x_0)$. As $q-x_0$ has the same parity as $y_0$ then we can conclude that $|B_p\cap C_q| = |B_q\cap C_p|$.
\end{proof}
\begin{lemma}
Let $p_1,\ldots, p_k$ be $k\ge 2$ pairwise coprime odd non-trivial divisors of $n$. Then
$$
2-2^{k-1}+\frac{2^{k-1}n}{p_1\ldots p_k}\le \left|\bigcap_{i=1}^kA_{p_i}\right|\le 2^{k-1}+\frac{2^{k-1}n}{p_1\ldots p_k}.
$$
\end{lemma}
\begin{proof}
Let $r=p_1\ldots p_k$. Notice that $\cap_{i=1}^kA_{p_i}=\cap_{i=1}^k{\left(B_{p_i}\cup C_{p_i}\right)}$. Hence
$$
\bigcap_{i=1}^kA_{p_i}=B_{r}\cup C_{r}\cup \bigcup_{p,q}\left(B_p\cap C_q\right)
$$
where $p$ and $q$ run through all products of fewer than $k$ of the terms $p_i$, such that $r=pq$ and $p$ and $q$ are coprime.

\smallskip

It is relatively easy to show that $B_r, C_r$ and each $B_p\cap C_q$ are mutually disjoint. To see this notice that if $m \in B_r\cap C_r$ then $m = rx = 1+ry$ and so $r|1$. If $m\in B_r\cap (B_p\cap C_q)$ then $m = rx = py = 1+qz$ and so $q|1$. If $m\in C_r\cap(B_p\cap C_q)$ then $m=1+rx = py = 1+qz$ and so $p|1$. Now let $m\in (B_p\cap C_q)\cap(B_q\cap C_p)$. Then $m=px=1+qy=qz=1+pw$ and so $p|1$. Finally if $\{p,q\}\ne\{s,t\}$ and if $r=pq = st$ and $m\in (B_p\cap C_q)\cap (B_s\cap C_t)$ then $m=px=1+qy = sz=1+tw$. Since $p\ne s$ then there exists $p_i|p$ and $p_i\ndiv s$ from which we deduce that $p_i\ndiv q$ and $p_i|t$. Hence we deduce that $p_i|1$.

\smallskip

Consequently, it follows from the previous lemma that
$$
|B_p\cap C_q| = \frac{n\pm r}{2r}
$$
and
$$
|C_r|=\frac{n+r}{2r}, |B_r|=\frac{n+r}{2r}.
$$
Hence
$$
 2^{k-1}\left(\frac{n}{r}-1\right)+2\le \left|\bigcap_{i=1}^kA_{p_i}\right|\le 2^{k-1}\left(\frac{n}{r}+1\right).
$$
Note it also follows that since $|B_p\cap C_q| = |B_q\cap C_p|$ then $\left|\bigcap_{i=1}^kA_{p_i}\right|$ is even.
\end{proof}

\medskip

Now suppose that $p_1,\ldots, p_\omega$ are the distinct odd prime divisors of $n$. Using the inclusion-exclusion principle, we see that 
$$
|A| = |\bigcup_{i=1}^\omega A_{p_i}| = \sum_{\emptyset\ne K\subseteq\{1,\ldots,\omega\}}{(-1)^{|K|-1}\left|\bigcap_{i\in K}A_{p_i}\right|},
$$
and from the above remarks, we can conclude that $|A|$ is even. Also from above, we see that an upper bound for $|A|$ is
$$
\sum_{i=1}^\omega{\left(\frac{n}{p_i}+1\right)} - \sum_{i_1\ne i_2\le \omega}{\left(\frac{2n}{p_{i_1}p_{i_2}}-2^1+2\right)}+\sum_{i_1,i_2,i_3\text{ distinct }\le\omega}\left(\frac{2^2n}{p_{i_1}p_{i_2}p_{i_3}}+2^2\right) - \ldots
$$
which, using equation~(\ref{-nSn-formula}) and the previously stated binomial series, simplifies to
$$
\left(\frac{n-S(n)}2\right)+ \left(\frac{3^\omega-1}{2}-2^\omega\right)+2.
$$
A lower bound is given by
$$
\sum_{i=1}^\omega{\left(\frac{n}{p_i}+1\right)} - \sum_{i_1\ne i_2\le \omega}{\left(\frac{2n}{p_{i_1}p_{i_2}}+2\right)}+\sum_{i_1,i_2,i_3\text{ distinct }\le\omega}\left(\frac{2^2n}{p_{i_1}p_{i_2}p_{i_3}}-2^2+2\right) - \ldots
$$
which simplifies to
$$
\left(\frac{n-S(n)}2\right) -\left(\frac{3^\omega-1}{2}-2^\omega\right).
$$
Hence we see that $T(n)$ satisfies the inequalities
$$
\left(\frac{S(n)-1}2\right)-\left(\frac{3^\omega-1}{2}-2^\omega+1\right) \le T(n)\le \left(\frac{S(n)-1}2\right)+\left(\frac{3^\omega-1}{2}-2^\omega+1\right).
$$
This completes the proof of Theorem~\ref{phi-double-dash-theorem}.
\end{proof}
Note that Theorem~\ref{primepower-theorem} is a corollary to this, and that $T(n)$ is odd if $n\equiv1\mod(4)$, while $T(n)$ is even if $n\equiv3\mod(4)$.
\begin{corollary}
If $n=p^eq^f$ with $p$ and $q$ distinct primes and $e,f\ge1$ then
$$
T(n) = \frac{S(n)-3}2\text{ or }T(n) = \frac{S(n)+1}2.
$$
\end{corollary}
\begin{proof}
This follows from the proof of Theorem~\ref{phi-double-dash-theorem} on observing that $|B_p\cap C_q| = |B_q\cap C_p|$. Given specific values for $p$ and $q$, it does not seem obvious how to predict which of the two values will emerge.
\end{proof}

\bigskip

We now consider the case when $n$ is even. First,
\begin{theorem}\label{power-2--theorem}
Let $e>1$ be an integer. Then $T(2)=0$ and 
$$
T(2^e) = 2^e-3\times2^{e-2} = 2^e\left(1-\frac34\right)=2^{e-2}.
$$
\end{theorem}
\begin{proof}
If $e>1$ then the number we need to count this time is
$$
(m\equiv0\mod 2)\lor (m\equiv1\mod 4).
$$
There are $2^{e-1}$ that satisfy the first condition and $2^{e-1}/2$ that satisfy the second condition. Hence the result.
\end{proof}
\begin{theorem}\label{4n-theorem}
Let $n>1$ be an odd integer. Then
$$
T(2^2n)=S(n).
$$
\end{theorem}
\begin{proof}
To compute $T(4n)$ we must remove from the set of residues $\{1,\ldots,4n\}$ numbers $m$ of the form
\begin{enumerate}
\item $m=2x$;
\item $m=xp$ where $p|n, p>1$ and $x$ is odd;
\item $m=1$ or $m=1+2xp$ where $p|n, p>1$ and $\gcd(m,n)=1$;
\item $m=1+4x$ where $\gcd(x,n)=1$ and $\gcd(m,n)=1$.
\end{enumerate}
If $1\le m\le 2n$ satisfies (2) or (3) then so does the number $m+2n$. If $m=1+4x$ with $\gcd(x,n)=\gcd(m,n)=1$ then $m+2n=3+4y$ with $\gcd(y,n)=\gcd(m+2n,n)=1$ and if $m=3+4x$ with $\gcd(x,n)=\gcd(m,n)=1$ then $m+2n = 1+4y$ with $\gcd(y,n)=\gcd(m+2n,n)=1$. Therefore for any number $m$ in the range  $1\le m\le 2n$ that satisfies (4), the number $m+2n$ does not satisfy (1)--(4) and should not be removed. Conversely, for any number $m$ in the range  $1\le m\le 2n$ that does not satisfy (1)--(4), the number $m+2n$ does satisfies (4) and should be removed. Consequently, to compute $T(4n)$ we can remove all even numbers, all odd numbers greater than $2n$ and those odd numbers less than $2n$ that satisfy (2) or (3). So the only numbers left are odd numbers less than $2n$ that do not satisfy (2) or (3).

We claim that the set of numbers left over has cardinality $S(n)$. Consider then the numbers $\{1,\ldots, n,\ n+1, \ldots, 2n\}$ and note that to compute $S(n)$ we remove from $\{1,\ldots, n\}$ those numbers $m$ such that both $m$ and $m-1$ are coprime to $n$. Then $m+n$ also has this property in the set $\{n+1,\ldots, 2n\}$ but with a different parity. Hence, the total number of remaining terms is $2n-2(n-S(n))=2S(n)$ and half of these are the odd numbers needed to compute $T(4n)$.
\end{proof}

\begin{theorem}\label{power-2-n-theorem}
Let $n>1$ be an odd integer and let $e>2$. Then
$$
T(2^en)=2T(2^{e-1}n)=2^{e-2}T(4n)=2^{e-2}S(n).
$$
\end{theorem}
\begin{proof}
Let $e\ge3$ and partition the set $\{1,\ldots, 2^en\}$ into two sets $\{1,\ldots, 2^{e-1}n\}$ and $\{1+2^{e-1}n,\ldots, 2^en\}$. For each element $x$ in the first set that contributes to $T(2^en)$, the element $x+2^{e-1}n$ in the second set also contributes to $T(2^en)$. The converse is clearly true as well. Hence $T(2^en) = 2T(2^{e-1}n)$, and the result will now follow by induction on $e$ and Theorem~\ref{4n-theorem}.
\end{proof}

\bigskip

The final case concerns integers of the form $2n$ for $n>1$ an odd positive integer. We proceed using techniques similar to those in Theorem~\ref{phi-double-dash-theorem}.

\smallskip

To compute $T(2n)$ we must remove from the set of residues $R=\{1,\ldots,2n\}$ numbers $m$ of the form
\begin{enumerate}
\item $m=2x$;
\item $m=xp$ where $p|n$ and $x$ is odd;
\item $m=1$ or $m=1+2xp$ where $p|n$ and $\gcd(m,n)=1$;
\item $m=1+4x$ where $\gcd(x,n)=\gcd(m,n)=1$.
\end{enumerate}
We first count the values that satisfy (4). As usual, in order to count those values which satisfy a given condition, we first count those values that don't.

\smallskip

Let $n>1$ be an odd positive integer, and let $p>1$ be an odd divisor of $n$. Define
\begin{gather*}
D_p=\{m|1\le m\le 2n, 4|(m-1)\text{ and } p|m\},\\
E_p=\{m|1\le m\le 2n, 4|(m-1)\text{ and } p|(m-1)\},\\
F_p = D_p\cup E_p.
\end{gather*}
\begin{lemma}\label{ap-lemma}
If $p>1$ is an odd divisor of $n$ then
$$
|E_p| = \frac{n-p}{2p}, |D_p|=\frac{n\pm p}{2p}
$$
and hence
$$
|F_p| = \frac np\text{ or }|F_p|=\frac np-1.
$$
\end{lemma}
\begin{proof}
Let $n>1$ be odd and let $p>1$ be an odd divisor of $n$.

\begin{enumerate}
\item Suppose that $1+4x=1+py\le 2n$ for positive integers $x$ and $y$. Then $y\le (2n-1)/p$ is congruent to 0 mod 4 and thus there are $(2n-2p)/4p=(n-p)/2p$ different candidates for $y$. Hence $|E_p| = (n-p)/2p$.

\item Now suppose that $1+4x = py\le 2n$ for positive integers $x$ and $y$. Then $py\le 2n$ and $py$ is congruent to 1 mod 4. Suppose that $p\equiv1\mod(4)$. Then the general solution to the equation $1+4x = py$ is $1+4(x_0+pt)=p(1+4t)$ where $p=1+4x_0$ and $t\in \z$. So $1+4t\le (2n-p)/p$ and so $0\le t\le (2n-2p)/4p = (n-p)/2p$. Hence there are $(n+p)/2p$ solutions in this case.

On the other hand, if $p\equiv3\mod(4)$ then the general solution is $1+4(x_0+pt)=p(3+4t)$ where $3p=1+4x_0$ and so $0\le4t\le 2n/p-6$ and there are $(n-3p)/2p+1 = (n-p)/2p$ solutions in this case.

Hence $|D_p| = (n\pm p)/2p$.
\end{enumerate}
Hence, for each $p$ there are either $n/p$ or $n/p-1$ elements in $F_p$.
\end{proof}

\medskip

Notice that if $p>1$ and $q>1$ are distinct odd divisors of $n$ that are coprime then $D_{pq} = D_p\cap D_q$ and $E_{pq} = E_p\cap E_q$.
\begin{lemma}
Let $p>1$ and $q>1$ be coprime odd divisors of $n$. Then
$$
\frac{2n}{pq}-2\le |F_p\cap F_q| \le \frac{2n}{pq}+1.
$$
\end{lemma}
\begin{proof}
If $m\in F_p\cap F_q = D_{pq}\cup E_{pq}\cup (D_q\cap E_p)\cup(D_p\cap E_q)$, then there are 4 possibilities
\begin{enumerate}
\item $1+4z=px=qy\le 2n$ for suitable $x,y,z \in\n$. Then $1+4z\in D_{pq}$ and so there are 
$$
|D_p\cap D_q|=|D_{pq}|=\frac{n\pm pq}{2pq}
$$
solutions.
\item $1+4z=1+px=1+qy\le2n$. Then $1+4z\in E_{pq}$ and so 
$$
|E_p\cap E_q| = |E_{pq}| = \frac{n- pq}{2pq}.
$$
\item $1+4z = 1+px = qy\le 2n$. From above, the solutions to $1+4z=1+px\le 2n$ are $z=tp$ for $1\le t\le (n-p)/2p$. 

Therefore we must solve $1+4pt = qy$ for $1\le t\le (n-p)/2p$. If $t_0, y_0$ is the smallest solution, then the general solution is
$$
1+4pt_0+4pqs = qy_0+4pqs
$$
for $s\in\z$. This means that $t_0\le q$ and so the interval $[1,(n-p)/2p]$ can be split into $(n-pq)/2pq$ `blocks' of $q$ consecutive integers with $(n-p)/2p-(n-pq)/2p=(q-1)/2$ integers left over. Each block of $q$ integers contains exactly one solution for $t$ and thus there are either $(n-pq)/2pq$ or $(n-pq)/2pq + 1 = (n+pq)/2pq$ solutions. Hence
$$
|E_p\cap D_q|=\frac{n\pm pq}{2pq}.
$$
Note that if $pq=n$ then there are zero `blocks' of $q$ integers, but there may still be a solution in the $(q-1)/2$ integers left over.
\item $1+4z = 1+qy=px\le 2n$. By symmetry, there are either $(n-pq)/2pq$ or $(n+pq)/2pq$ solutions. Hence
$$
|D_p\cap E_q|=\frac{n\pm pq}{2pq}.
$$
\end{enumerate}
The result then follows.
\end{proof}
\begin{lemma}
Let $p_1,\ldots, p_k$ be $k\ge 2$ pairwise coprime odd non-trivial divisors of $n$. Then
$$
-2^{k-1}+\frac{2^{k-1}n}{p_1\ldots p_k}\le \left|\bigcap_{i=1}^kF_{p_i}\right|\le 2^{k-1}-1+\frac{2^{k-1}n}{p_1\ldots p_k}.
$$
\end{lemma}
\begin{proof}
Let $r=p_1\ldots p_k$. Notice that $\cap_{i=1}^kF_{p_i}=\cap_{i=1}^k{\left(D_{p_i}\cup E_{p_i}\right)}$. Hence
$$
\bigcap_{i=1}^kF_{p_i}=D_{r}\cup E_{r}\cup \bigcup_{p,q}\left(D_p\cap E_q\right)
$$
where $p$ and $q$ run through all products of fewer than $k$ of the terms $p_i$, such that $r=pq$ and $p$ and $q$ are coprime. From the previous lemma, we have
$$
|D_p\cap E_q| = \frac{n\pm r}{2r}
$$
and
$$
|E_r|=\frac{n-r}{2r}, |D_r|=\frac{n\pm r}{2r}.
$$
Hence
$$
 \frac{2^k}2\left(\frac{n}{r}-1\right)\le \left|\bigcap_{i=1}^kF_{p_i}\right|\le \frac{2^k-2}2\left(\frac{n}{r}+1\right)+\frac{n}{r}.
$$
\end{proof}

\bigskip

Suppose that $n$ has $\omega$ distinct odd prime factors $p_1,\ldots, p_\omega$. Because there are $(2n-2)/4 = (n-1)/2$ values congruent to 1 mod 4, and excluding 1, in the set $\{1,\ldots,2n\}$, then there are
$$
\frac{n-1}2-\left|\bigcup_{i=1}^\omega F_{p_i}\right|
$$
values $m$, excluding 1, which are congruent to 1 mod 4 but such that neither $m$ nor $m-1$ has a factor in common with $n$. Using the inclusion-exclusion principle, we see that this is bounded above by 
$$
\frac{n-1}2-\sum\left(\frac{n}{p_i}-1\right)+\sum\left(\frac{2n}{p_ip_j}+1\right)-\ldots
$$
which, using equation~(\ref{-nSn-formula}), simplifies to
$$
-\frac 12+\frac n2\prod_{i=1}^\omega{\left(1-\frac2{p_i}\right)}+\omega+(2^1-1)\binom \omega2+2^2\binom \omega3+(2^3-1)\binom \omega4+\ldots
$$
So the upper bound is
$$
\frac12\left(S(n)-1\right) +\omega+2^1\binom \omega2+2^2\binom \omega3+2^3\binom \omega4+\ldots  -\binom \omega2-\binom \omega4 - \ldots
$$
or
$$
\frac12\left(S(n)-1\right) + \frac 12 3^\omega -\frac 12-\frac 12\left(2^\omega-2\right)=\frac12\left(S(n)+3^\omega-2^\omega\right).
$$
The lower bound is
$$
\frac{n-1}2-\sum\left(\frac{n}{p_i}\right)+\sum\left(\frac{2n}{p_ip_j}-2\right)-\ldots
$$
which simplifies to
$$
\frac12\left(S(n)-1\right)-\frac12\left(2^2\binom \omega2+2^3\binom \omega3+2^4\binom \omega4+\ldots\right)+\binom \omega3+\binom \omega5+\ldots
$$
or
$$
\frac12\left(S(n)-1\right)-\frac12\left(3^\omega-2\omega-1\right)+\frac12\left(2^\omega-2\omega\right)=\frac12\left(S(n)-(3^\omega-2^\omega)\right).
$$
\begin{theorem}\label{2n-theorem}
Let $n>1$ be an odd integer with $\omega$ distinct prime divisors. Then
$$
\left|T(2n)-\frac{S(n)-1}2\right|\le\frac{3^\omega-2^\omega+1}2.
$$
If $n=p^e$ where $p$ is prime and $e>0$, then $T(2n) = (S(n)-1)/2$ when $p\equiv3\mod(4)$ and $T(2n) = (S(n)+1)/2$ when $p\equiv1\mod(4)$.
\end{theorem}
\begin{proof}
Recall that to compute $T(2n)$ we need to remove from the set of residues $R=\{1,\ldots,2n\}$ numbers $m$ of the form
\begin{enumerate}
\item $m=2x$;
\item $m=xp$ where $p|n$ and $x$ is odd;
\item $m=1$ or $m=1+2xp$ where $p|n$ and $\gcd(m,n)=1$;
\item $m=1+4x$ where $\gcd(x,n)=\gcd(m,n)=1$.
\end{enumerate}
Partition the set $R$ into $n$ disjoint subsets $R_i = \{i,n+i\}$ for $1\le i\le n$. Identify those sets $R_j$ for which - either $j=1$, $\gcd(j,n)\ne1$ or $\gcd(j-1,n)\ne 1$. There are $n-S(n)$ such subsets. If $j$ is odd, then $j$ satisfies either (2) or (3). If $j$ is even, then $n+j$ satisfies either (2) or (3). Hence, all the elements in the $R_j$ satisfy either (1), (2), or (3) and should be removed. For the subsets $R_k$ that are left, half of the elements are even and the other half are odd and do not satisfy (1), (2), or (3). There are then $S(n)$ such odd terms, some of which may satisfy (4). From above, the number of these $S(n)$ terms satisfying (4) lies in the range
$$
\left[\frac12\left(S(n)-\left(3^\omega-2^\omega\right)\right),\frac12\left(S(n)+\left(3^\omega-2^\omega\right)\right)\right]
$$
and so $T(2n)$ lies in the range
$$
\left[\frac12\left(S(n)-\left(3^\omega-2^\omega\right)\right),\frac12\left(S(n)+\left(3^\omega-2^\omega\right)\right)\right].
$$
The final part follows from (the proof of) Lemma~\ref{ap-lemma}.
\end{proof}

\bigskip

To summarise, if $n = 2^em$ for $e\ge 0$ and $m>1$ is odd and if $\omega$ is the number of distinct prime factors of $m$, then
\begin{enumerate}
\item $T(1) = T(2) = 0$ and $T(2^e) = 2^{e-2}$ for $e\ge2$;
\item $T(n) = 2^{e-2}S(m)$ for $e\ge 2$;
\item $\left|T(n)-\frac{S(m)-1}2\right|\le(3^\omega-2^\omega+1)/2$ for $e=1$;
\item $\left|T(n)- \frac{S(m)-1}2\right|\le(3^\omega-2^{\omega+1}+1)/2$ when $e=0$;
\item $T(n) = (S(m)-1)/2$ if $e=1$ and $m=p^f, f>0$ where $p\equiv3\mod 4$ is prime;
\item $T(n) = (S(m)-1)/2+1$ if $e=1$ and $m=p^f, f>0$ where $p\equiv1\mod 4$ is prime;
\item $T(n) = (S(m)-1)/2$ when $e=0$ and $\omega=1$;
\item $T(n) = (S(m)-1)/2\pm1$ when $e=0$ and $\omega=2$.
\end{enumerate}

\section{Discrete Logs and variants of groups}\label{motivation-section}

In this section we provide a more detailed account of the cryptographic scheme mention in Section~\ref{introduction-section}, which was our motivation for introducing the totient function $T(n)$.

\medskip

Let $G$ be a finite group of order $n$. If $g\in G$ and $e \in U_{n}$ then there exists $e'\in U_{n}$ such that $ee'\equiv1\mod{n}$ and thus
$$
g = g^{ee'} = (g^e)^{e'}.
$$
In cryptographic terms, we refer to $g$ as the {\em plaintext}, $e$ as the {\em key}, and $g^e$ as the {\em ciphertext}. We also refer to $e$ as the {\em discrete log} of $g^e${\em\ mod }$g$. The {\em discrete log problem} is the problem of determining $e$ given both $g$ and $g^e$, and many cryptographic systems rely on this problem being a hard one to solve for certain suitable groups $G$.

If we are reduced to a `brute force' solution of the discrete log problem, then we potentially have to test every candidate for $e$ from the set $U_n$. Now, by definition, $|U_n| = \phi(n)$ and if we choose the group carefully then we can expect that $\phi(n) \sim O(n)$ and thus the problem becomes infeasible for large enough values of $n$.

\smallskip

A slightly different approach, with potentially increased security, is to use a {\em variant} of the group. Variants were introduced in semigroup theory by Hickey~\cite{hickey-83,hickey}, based on the earlier work of Chase~\cite{chase-1,chase-2}, on variants of semigroups of binary relations, and a number of interesting properties and types of variants of semigroups have been studied. If $(S,\cdot)$ is a semigroup and $s\in S$ then we can define a new multiplication, $\ast_s$, on $S$ by
$$
x\ast_s y = x\cdot s\cdot y.
$$
It is easy to check that this gives an associative operation, and so the system $(S,\ast_s)$ is a semigroup, referred to as a {\em variant of }$S$ and often denoted by $S^s$.

A group is just a special kind of semigroup, and if $(G,\cdot)$ is a group and $x\in G$ then it is easy to check that the system $G^x=(G,\ast_x)$ is also a group, with identity $x^{-1}$ and where the inverse of $g$, in $G^x$, is given by the element $x^{-1}\cdot g^{-1}\cdot x^{-1}$ in $G$, and where $x^{-1}$ and $g^{-1}$ are the inverses of $x$ and $g$ in the {\em base} group $(G,\cdot)$. It is also easy to see that the map $(G,\cdot)\to (G,\ast_x)$ given by $g\mapsto g\cdot x^{-1}$ is a group isomorphism (this is not generally the case for semigroups). It clearly follows that if $e\in\n$ then within $G^x$, the element $g^e$, the $e^{\text{th}}$ power of the element $g$, is represented by the element $\left(g\cdot x\right)^{e-1}\cdot g$ in $(G,\cdot)$. To simplify notation, we shall subsequently omit the symbol `$\cdot$' and write this as $\left(gx\right)^{e-1}g$.

The discrete log problem in the variant $G^x$ is then the problem of finding $e$ given both $g$ and $c=(gx)^{e-1}g$. If, however, $x$ were kept `secret' then this is a different problem to the standard case (if $x$ were known, then we can compute both $gx$ and $(gx)^e$ which then reduces to the `standard' discrete log problem).Thus the idea is to {\em encrypt} $g$ as $\left(gx\right)^{e-1}g$ using the pair $(x,e)$ as the key. The recipient of this ciphertext can decipher by first using the Euclidean algorithm to calculate $e'\in U_n$ such that $ee'\equiv1\mod(n)$, and then compute
$$
\left(cx\right)^{e'}x^{-1} = \left(\left(gx\right)^{e-1}gx\right)^{e'}x^{-1} = \left(gx\right)^{ee'}x^{-1} = gxx^{-1}=g.
$$

If an attacker had no other information other than $g$ and $c$, then a `brute force' attack of this discrete log problem would involve computing all values of the form $(gy)^{f-1}$, where $y$ runs through all values in $G$ and $f$ runs through all values in $U_n$. There are then $n\phi(n)\sim O(n^2)$ such pairs, effectively doubling the security.

\smallskip

However, there is another issue to consider: the question of whether there are any `false positives' that will emerge. In other words, is it possible for $(gx)^{e-1} = (gy)^{f-1}$ for $(x,e)\ne (y,f)$? We can translate this problem as follows. Given a finite group $G$ of order $n$ and a fixed (constant) element $c\in G$, we wish to find the number of units $f>1$ in $\z_n$, with the property that the equation
\begin{equation}\label{mimic-equation}
w^{f-1}=c
\end{equation}
has a solution in $G$, given that there is at least one value of $f$ and $w$ solving the equation.

\begin{theorem}\label{tn-theorem}
Let $G$ be a finite group of order $n$ and let $c$ be a fixed constant in $G$ and assume that Equation~\ref{mimic-equation} has a solution for at least one value of $f$. If $n$ is odd then there are at least $S(n)$ solutions to Equation~\ref{mimic-equation} while if $n$ is even there are at least $T(n)$ solutions.
\end{theorem}
\begin{proof}
Suppose $n$ is odd and suppose that $f-1$ is also a unit in $\z_n$ with inverse $k$. Then $w=c^k$ is a solution since $\left(c^k\right)^{f-1} = c$. There are, by definition, $S(n)$ such units in $\z_n$.

\smallskip

On the other hand, if $n$ is even then $f$ must be odd and therefore, if there is a solution, then $c=h^2$ for some $h\in G$. In this case, suppose that $(f-1)/2$ is also a unit in $\z_n$ with inverse $k$. Then $w=h^k$ is a solution of Equation~(\ref{mimic-equation}) since
$$
\left(h^k\right)^{f-1} = \left(h^k\right)^{2(f-1)/2} = \left(h^2\right)^{k(f-1)/2} = h^2 = c.
$$
Thus, there are at least $T(n)$ such solutions.
\end{proof}

\smallskip

So by using variants of groups, we can effectively double the security of a simple brute-force attack on the discrete log problem, and at the same time, by choosing a suitably large value of $n$ with $S(n)$ and $T(n)$ large enough, make it infeasibly hard to discover the key. 

\smallskip

For the specific example that we gave in Section~\ref{introduction-section}, the situation is even more interesting.

\begin{proposition}
If $G=U_p$ where $p=2q+1$ is a safe prime, then there are $p-5=2(q-2)$ solutions to Equation~\ref{mimic-equation}.
\end{proposition}
\begin{proof}
By Theorem~\ref{tn-theorem} we see that each unit $f>1$ with the property that $(f-1)/2$ is also a unit, provides a solution, $w$ say. Notice that in this case, $-w$ is also a solution. If $(f-1)/2$ is not a unit it is because it is even, in which case $(f-1)/4$ may be a unit. If not then $(f-1)/4$ is even and so $(f-1)/8$ may be a unit. Continuing in this fashion we see that there is a positive integer $m$ such that $(f-1)/2^m$ is a unit.

\smallskip

Now it is well known that if $p\equiv3\mod(4)$ is a prime and if $y\in\z_p$ then either $y$ or $-y$, but not both, has a square root modulo $p$ (the square root is in fact $y^{(p+1)/4}$). If in Equation~\ref{mimic-equation}, $c=h^2$ then either $h$ or $-h$ will have a square root, $h_1$ say, and so $c = h_1^4$. But then either $h_1$ or $-h_1$ will have a square root, $h_2$ say, and so $c = h_2^8$. Continuing in this fashion we see that $c=h_{m-1}^{2^m}$. So if $k$ is the multiplicative inverse of $(f-1)/2^{m}$ then $w=\pm h_{m-1}^k$ will be solutions to Equation~\ref{mimic-equation}. Hence every unit $f>1$ in $\z_p$ provides two solutions and since $|U_{p-1}| = q-1$ the result follows.
\end{proof}

\smallskip

In practice, there are certain {\em known plaintext attacks} that can sometimes reduce this discrete log problem to the more classic case and choosing a random value for $x$ for each encryption will help mitigate this. For example, choose $e\in U_n$ and $s\in \z_m$ for some $m\ge n$ and let the pair $(s,e)$ be the secret key. Then given plaintext $g\in G$, let $i$ be a random value in $\z_m$ and define $x_i = H(i\oplus s)$ where $\oplus$ is the bitwise XOR operator and $H$ is a suitably chosen cryptographic hash function whose image coincides with $G$. The ciphertext is then the pair $(i,(gx_i)^{e-1}g)$, and anyone with access to the key, can replicate $x_i$ and decrypt as before. However, even if an attacker can identify the correct value of $x_i$ amongst all the $T(n)$ or more pairs from Theorem~\ref{tn-theorem}, it will be relatively ineffective as we shall use a different value of $x_i$ for each encryption.

\medskip

Note that if $(m-1)$ is divisible by $r$ for some $r>1$ then a unit $m=1+rs$ with the property that $s=(m-1)/r$ is also a unit, could provide us with a potentially different solution to Equation~(\ref{mimic-equation}), suggesting that a study of those functions $T_r(n)$ that count the number of units $m\in U_n$ such that $(m-1)/r\in U_n$ might also be fruitful.

\smallskip

Finally, with the proof of Lemma~\ref{odd-intersect-lemma} in mind, we can consider the following problem. If $p$ and $q$ are coprime odd positive integers, then the Euclidean algorithm allows us to find positive integers $a$ and $b$ such that $ap-bq=1$ and such that $a$ and $b$ are the smallest such positive integers. We can easily observe that $a$ and $b$ must have different parity. Given such $p$ and $q$, can we predict which of $a$ and $b$ is even?

\bigskip

\bibliographystyle{vancouver}

\begin{thebibliography}{00}
\bibitem{chase-1} Karen Chase, Sandwich semigroups of binary relations, {\it Discrete Math.}, 28 (1979), 231-236.

\bibitem{chase-2} Karen Chase, New semigroups of binary relations, {\it Semigroup Forum}, 18 (1979), 79-82.
 
\bibitem{hickey-83} J. B. Hickey, Semigroups under a sandwich operation, {\it Proc. Edinburgh Math. Soc.}, 26 (1983), 371-382.
 
\bibitem{hickey} J.B. Hickey, On variants of a semigroup, {\it Bulletin of the Australian Mathematical Society},Vol 34, Issue 3, (1986), 447--459.

\bibitem{hoffstein} Jeffrey Hoffstein, Jill Catherine Pipher and Joseph H. Silverman, {\it An introduction to mathematical cryptography}, (Springer 2014)

\bibitem{jones-2005} Gareth A. Jones and J. Mary Jones {\it Elementary Number Theory}, Springer Undergraduate Maths Series, (Springer, 2005).

\bibitem{schemmel} Schemmel, V. \"Uber relative Primzahlen, Journal f\"ur die reine und angewandte Mathematik, Band 70 (1869), S. 191–192.
\end{thebibliography}

\end{document}